\documentclass[a4paper,11pt,twoside]{article}	
\usepackage{amsmath,amssymb,amsthm,amsfonts,color,graphicx}
\usepackage{setspace}
\usepackage[margin=4cm]{geometry}
\newtheorem{theorem}{Theorem}

\DeclareMathOperator{\rr}{\mathbb R}

\newtheorem{proposition}{Proposition}
\newtheorem{remark}{Remark}
\newtheorem{claim}{Claim}

\theoremstyle{definition}

\numberwithin{equation}{section}
\usepackage{layout,textcomp}
\newtheorem*{Ak}{Acknowledgements}
\title{PERIODIC MINIMAL SURFACES IN SEMIDIRECT PRODUCTS}
\author{ANA MENEZES}
\date{}

\begin{document}
\maketitle

\begin{abstract}
In this paper we prove existence of complete minimal surfaces in some metric semidirect products. These surfaces are similar to the doubly and singly periodic Scherk minimal surfaces in $\mathbb R^3.$ In particular, we obtain these surfaces in the Heisenberg space with its canonical metric, and in Sol$_3$ with a one-parameter family of non-isometric metrics.
\end{abstract}

\vspace{0.5cm}


\vspace{0.5cm}


\section{Introduction}
In this paper we construct examples of periodic minimal surfaces in some semidirect products $\mathbb R^2\rtimes_A \mathbb R,$ depending on the matrix $A.$ By periodic surface we mean a properly embedded surface invariant by a nontrivial discrete group of isometries.

One of the most simple examples of semidirect product is $\mathbb H^2\times\mathbb{R}=\mathbb R^2\rtimes_{A}\mathbb R,$ when we take $A=\left(\begin{array}{ccc}	1 & 0\\	0&0 \end{array}\right).$ In this space, Mazet, Rodr\' iguez and Rosenberg \cite{MMR} proved some results about periodic constant mean curvature surfaces and they constructed examples of such surfaces. One of their methods is to solve a Plateau problem for a certain contour.  In \cite{R}, using a similar technique, Rosenberg constructed examples of complete minimal surfaces in $M^2\times \mathbb R,$ where $M$ is either the two-sphere or a complete Riemannian surface with nonnegative curvature or the hyperbolic plane.

Meeks, Mira, P\' erez and Ros \cite{MMPR} have proved results concerning the geometry of solutions to Plateau type problems in metric semidirect products $\mathbb R^2\rtimes_A \mathbb R,$ when there is some geometric constraint on the boundary values of the solution (see Theorem \ref{th-MMPR}).

The first example that we construct is a complete periodic minimal surface similar to the doubly periodic Scherk minimal surface in $\mathbb R^3$. It is invariant by two translations that commute and is a four punctured sphere in the quotient of $\mathbb R^2\rtimes_A \mathbb R$ by the group of isometries generated by the two translations. In the last section we obtain a complete periodic minimal surface analogous to the singly periodic Scherk minimal surface in $\mathbb R^3.$ 

These surfaces are obtained by solving the Plateau problem for a geodesic polygonal contour $\Gamma$ (it uses a result by Meeks, Mira, P\'erez and Ros \cite{MMPR} about the geometry of solutions to the Plateau problem in semidirect products), and letting some sides of $\Gamma$ tend to infinity in length, so that the associated Plateau solutions all pass through a fixed compact region (this will be assured by the existence of minimal annuli playing the role of barriers). Then a subsequence of the Plateau solutions will converge to a minimal surface bounded by a geodesic polygon with edges of infinite length. We complete this surface by symmetry across the edges. The whole construction requires precise geometric control and uses curvature estimates for stable minimal surfaces. 

These results are obtained for semidirect products $\mathbb R^2\rtimes_A \mathbb R$ where $A=\left(\begin{array}{ccc}	0 & b\\
	c&0
\end{array}\right).$ For example, we obtain periodic minimal surfaces in the Heisenberg space, when $A=\left(\begin{array}{ccc}	0 & 1\\
	0&0
\end{array}\right),$ and in Sol$_3,$ when $A=\left(\begin{array}{ccc}	0 & 1\\
	1&0
\end{array}\right),$ with their well known Riemannian metrics. When we consider the one-parameter family of matrices $A(c)= \left(
\begin{array}{ccc}
	0 & c\\
	\frac{1}{c}&0
\end{array}\right), c\geq 1,$ we get a one-parameter family of metrics in Sol$_3$ which are not isometric.

{\small
\begin{Ak} This work is part of the author's Ph.D. thesis at IMPA. The author would like to express her sincere gratitude to her advisor Prof. Harold Rosenberg for his constant encouragement and guidance throughout the preparation of this work. The author would also like to thank Joaqu\' in P\' erez for helpful conversations about semidirect products. The author was financially suported by CNPq-Brazil and IMPA.
\end{Ak}}

\section{Preliminary Results}

Generalizing direct products, a semidirect product is a particular way in which a group can be constructed from two subgroups, one of which is a normal subgroup. As a set, it is the cartesian product of the two subgroups but with a particular multiplication operation. 

In our case, the normal subgroup is $\mathbb R^2$ and the other subgroup is $\mathbb R.$ Given a matrix $A \in \mathcal M_2(\mathbb R),$ we can consider the semidirect product $\mathbb R^2\rtimes_A \mathbb R,$ where the group operation is given by
\begin{equation}
(p_1,z_1)\ast(p_2,z_2)=(p_1+{\rm e}^{z_1A}p_2, z_1+z_2), \ p_1,p_2\in \mathbb R^2, z_1,z_2\in \mathbb R \label{eq1}
\end{equation}
and
$$
A=\left(
\begin{array}{ccc}
	a & b\\
	c&d
\end{array}\right)\in \mathcal M_2(\mathbb R).
$$

We choose coordinates $(x,y)\in \mathbb R^2, z\in \mathbb R.$ Then $\partial_x=\frac{\partial}{\partial x}, \partial_y, \partial_z$ is a parallelization of $G=\mathbb R^2\rtimes_A \mathbb R.$ Taking derivatives at $t=0$ in (\ref{eq1}) of the left multiplication by $(t,0,0)\in G$ (respectively by $(0,t,0), (0,0,t)$), we obtain the following basis $\{F_1, F_2, F_3\}$ of the right invariant vector fields on $G$:
\begin{equation}
\label{eq2}
F_1=\partial_x,\ F_2= \partial_y,\ F_3=(ax+by)\partial_x+(cx+dy)\partial_y+\partial_z.
\end{equation}


Analogously, if we take derivatives at $t=0$ in (\ref{eq1}) of the right multiplication by $(t,0,0)\in G$ (respectively by $(0,t,0), (0,0,t)$), we obtain the following basis $\{E_1, E_2, E_3\}$ of the Lie algebra of $G$:
\begin{equation}
\label{eq3}
E_1=a_{11}(z)\partial_x+a_{21}(z)\partial_y,\ E_2=a_{12}(z)\partial_x+a_{22}(z)\partial_y, \ E_3=\partial_z,
\end{equation}
where we have denoted
$${\rm e}^{zA}=\left(
\begin{array}{ccc}
a_{11}(z)&a_{12}(z)\\
a_{21}(z)&a_{22}(z)
\end{array}
\right).$$

We define the \textit{canonical left invariant metric} on $\mathbb R^2\rtimes_A \mathbb R$, denoted by $\left\langle, \right\rangle,$ to be that one for which the left invariant basis $\{E_1,E_2,E_3\}$ is orthonormal.

The expression of the Riemannian connection $\nabla$ for the canonical left invariant metric of $\mathbb R^2\rtimes_A \mathbb R$ in this frame is the following:
$$\begin{array}{lll}
\nabla_{E_1}E_1=aE_3&\nabla_{E_1}E_2=\frac{b+c}{2}E_3&\nabla_{E_1}E_3=-aE_1-\frac{b+c}{2}E_2\\
&&\\
\nabla_{E_2}E_1=\frac{b+c}{2}E_3&\nabla_{E_2}E_2=dE_3&\nabla_{E_2}E_3=-\frac{b+c}{2}E_1-dE_2\\
&&\\
\nabla_{E_3}E_1=\frac{c-b}{2}E_2&\nabla_{E_3}E_2=\frac{b-c}{2}E_1&\nabla_{E_3}E_3=0.
\end{array}
$$

In particular, for every $(x_0,y_0)\in \mathbb R^2, \gamma(z)=(x_0,y_0,z)$ is a geodesic in $G.$

\begin{remark}
As $[E_1,E_2]=0,$ thus for all $z,$ $\mathbb R^2\rtimes_A \{z\}$ is flat and the horizontal straight lines are geodesics. Moreover, the mean curvature of $\mathbb R^2\rtimes_A \{z\}$ with respect to the unit normal vector field $E_3$ is the constant $H=\mbox{tr}(A)/2.$
\label{rem1}
\end{remark}

The change from the orthonormal basis $\{E_1,E_2,E_3\}$ to the basis $\{\partial_x, \partial_y, \partial_z\}$ produces the following expression for the metric $\left\langle ,\right\rangle:$
$$
\begin{array}{rcl}
\left\langle , \right\rangle_{(x,y,z)} &=& \displaystyle [a_{11}(-z)^2+a_{21}(-z)^2]dx^2+[a_{12}(-z)^2+a_{22}(-z)^2]dy^2+dz^2\\
&&\\
&&+[a_{11}(-z)a_{12}(-z)+a_{21}(-z)a_{22}(-z)](dx\otimes dy+dy\otimes dx)\\
&&\\
&=&\displaystyle {\rm e}^{-2\mbox{tr}(A)z}\displaystyle \{[a_{21}(z)^2+a_{22}(z)^2]dx^2+[a_{11}(z)^2+a_{12}(z)^2]dy^2\} +dz^2\\
&&\\
&&-{\rm e}^{-2\mbox{tr}(A)z}[a_{11}(z)a_{21}(z)+a_{12}(z)a_{22}(z)](dx\otimes dy+dy\otimes dx).
\end{array}
$$

In particular, for every matrix $A\in\mathcal M_2(\mathbb R),$ the rotation by angle $\pi$ around the vertical geodesic $\gamma(z)=(x_0,y_0,z)$ given by the map $R(x,y,z)=(-x+2x_0,-y+2y_0, z)$ is an isometry of $(\mathbb R^2\rtimes_A \mathbb R,\left\langle , \right\rangle)$ into itself. 

\begin{remark}
As we observed, the vertical lines of $\mathbb R^2\rtimes_A \mathbb R$ are geodesics of its canonical metric. For any line $l$ in $\mathbb R^2\rtimes_A \{0\}$ let $P_l$ denote the vertical plane $\{(x,y,z): (x,y,0)\in l; z\in \mathbb R\}$ containing the set of vertical lines passing through $l.$ It follows that $P_l$ is ruled by vertical geodesics and, since rotation by angle $\pi$ around any vertical line in $P_l$ is an isometry that leaves $P_l$ invariant, then $P_l$ has zero mean curvature.
\end{remark}

Although the rotation by angle $\pi$ around horizontal geodesics is not always an isometry, we have the following result.

\begin{proposition}
Let $A=\left(\begin{array}{cc} 0&b\\c&0\end{array}\right)\in\mathcal M_2(\mathbb R)$ and consider the horizontal geodesic $\alpha=\{(x_0,t,0): t\in \mathbb R\}$ in $\mathbb R^2\rtimes_A \{0\}$ parallel to the $y$-axis. Then the rotation by angle $\pi$ around $\alpha$ is an isometry of $(\mathbb R^2\rtimes_A \mathbb R, \left\langle , \right\rangle)$ into itself. The same result is true for a horizontal geodesic parallel to the $x$-axis.
\label{iso}
\end{proposition}

\begin{proof}
The rotation by angle $\pi$ around $\alpha$ is given by the map $\phi(x,y,z)=(-x+2x_0,y, -z),$ so $\phi_x=-\partial_x, \ \phi_y=\partial_y$ and $\phi_z=-\partial_z.$

If $A=\left(\begin{array}{cc} 0&b\\c&0\end{array}\right)$, then
$$
{\rm e}^{zA}= \displaystyle \left(\begin{array}{ll}  \displaystyle \sum_{k=0}^{\infty}{\frac{(bc)^kz^{2k}}{(2k)!}}&\displaystyle\sum_{k=1}^{\infty}{\frac{b^kc^{k-1}z^{2k-1}}{(2k-1)!}}\\
\displaystyle\sum_{k=1}^{\infty}{\frac{c^kb^{k-1}z^{2k-1}}{(2k-1)!}}&\displaystyle\sum_{k=0}^{\infty}{\frac{(bc)^kz^{2k}}{(2k)!}}\end{array}\right).
$$

Hence, $a_{11}(z)=a_{22}(z)$ and ${\rm e}^{-zA}=\left(\begin{array}{cc} a_{11}(z)&-a_{12}(z)\\-a_{21}(z)&a_{11}(z)\end{array}\right).$ Then
$$
\begin{array}{rcl}
\left\langle , \right\rangle_{(x,y,z)}&=&\displaystyle \{[a_{21}(z)^2+a_{11}(z)^2]dx^2+[a_{11}(z)^2+a_{12}(z)^2]dy^2\} +dz^2\\
&&\\
&&-[a_{11}(z)a_{21}(z)+a_{12}(z)a_{11}(z)](dx\otimes dy+dy\otimes dx)
\end{array}
$$
and
$$
\begin{array}{rcl}
\left\langle , \right\rangle_{\phi(x,y,z)}&=& \displaystyle \{[a_{21}(z)^2+a_{11}(z)^2]dx^2+[a_{11}(z)^2+a_{12}(z)^2]dy^2\} +dz^2\\
&&\\
&&+[a_{11}(z)a_{21}(z)+a_{12}(z)a_{11}(z)](dx\otimes dy+dy\otimes dx).
\end{array}
$$

Therefore, $\left\langle\phi_x,\phi_x \right\rangle_{\phi(x,y,z)}=\left\langle \partial_x,\partial_x\right\rangle_{(x,y,z)},\ \left\langle\phi_y,\phi_y \right\rangle=\left\langle \partial_y,\partial_y\right\rangle,$ \ $\left\langle\phi_z,\phi_z \right\rangle=\left\langle\partial_z,\partial_z \right\rangle,$ that is, $\phi$ is an isometry. Analogously, we can show that the rotation by angle $\pi$ around a horizontal geodesic parallel to the $x$-axis is also an isometry.
\end{proof}


%
\begin{remark}
When the matrix $A$ in $\mathbb R^2\rtimes_A \mathbb R$ is $\left(
\begin{array}{ccc}
	0 & 1\\
	0&0
\end{array}\right)$ and  $\left(
\begin{array}{ccc}
	0 & 1\\
	1&0
\end{array}\right),$ we have the Heisenberg space and Sol$_{3},$ respectively, with their well known Riemannian metrics. When we consider the one-parameter family of matrices $A(c)= \left(
\begin{array}{ccc}
	0 & c\\
	\frac{1}{c}&0
\end{array}\right), c\geq 1,$ we get a one-parameter family of metrics in Sol$_3$ which are not isometric. For more details, see \cite{MP}.
\end{remark}

Meeks, Mira, P\' erez and Ros \cite{MMPR} have proved results concerning the geometry of solutions to Plateau type problems in metric semidirect products $\mathbb R^2\rtimes_A \mathbb R,$ when there is some geometric constraint on the boundary values of the solution. More precisely, they proved the following theorem.

\begin{theorem}[Meeks, Mira, P\'erez and Ros, \cite{MMPR}]
Let $X= \mathbb R^2 \rtimes_A \mathbb R$ be a metric semidirect product with its canonical metric and let $\Pi: \mathbb R^2 \rtimes_A \mathbb R \rightarrow \mathbb R^2 \rtimes_A \{0\}$ denote the projection $\Pi(x,y,z)=(x,y,0).$ Suppose $E$ is a compact convex disk in $\mathbb R^2 \rtimes_A \{0\}$, $C=\partial E$ and $\Gamma \subset \Pi^{-1}(C)$ is a continuous simple closed curve such that $\Pi: \Gamma \rightarrow C$ monotonically parametrizes $C.$ Then,
\label{th-MMPR}
\begin{enumerate}
\item $\Gamma$ is the boundary of a compact embedded disk $\Sigma$ of finite least area.
\item The interior of $\Sigma$ is a smooth $\Pi$-graph over the interior of $E.$
\end{enumerate}
\end{theorem}

\section{A doubly periodic Scherk minimal surface}
Throughout this section, we consider the semidirect product $\mathbb R^2\rtimes_A \mathbb R$ with the canonical left invariant metric $\left\langle, \right\rangle,$ where $A=\left(\begin{array}{cc} 0&b\\c&0\end{array}\right).$ In this space, we prove the existence of a complete minimal surface analogous to Scherk's doubly periodic minimal surface in $\mathbb R^3$.

Fix $0<c_0< c_1$ and let $a$ be a sufficiently small positive quantity such that
\begin{equation}
\begin{array}{rcl}
a &<& \displaystyle\int_{c_0}^{c_1}{\sqrt{a_{11}^2(z)+a_{21}^2(z)}+\sqrt{a_{11}^2(z)+a_{12}^2(z)}}{\rm d}z\\
&&\\
&&-\displaystyle\int_{c_0}^{c_1}{\sqrt{(a_{11}(z)+a_{12}(z))^2+(a_{11}(z)+a_{21}(z))^2}}{\rm d}z.
\label{eqdefa}
\end{array}
\end{equation}
Note that such positive number $a$ exists, as $|\partial_x|=\sqrt{a_{11}^2(z)+a_{21}^2(z)},$ \ $|\partial_y|=\sqrt{a_{11}^2(z)+a_{12}^2(z)}$ and  $|\partial_x +\partial_y|=\sqrt{(a_{11}(z)+a_{12}(z))^2+(a_{11}(z)+a_{21}(z))^2}.$

For each $c>0,$ consider the polygon $P_c$ in $\mathbb {R}^2 \rtimes_A \mathbb{R}$ with the sides $\alpha_1,\alpha_2,\alpha_3^c,\alpha_4^c$ and $\alpha_5^c$ defined below.
$$\begin{array}{rcl}
\alpha_1 &=& \left\{(t,0,0): 0\leq t \leq a\right\}\\
&&\\
\alpha_2&=&\left\{(0,t,0): 0\leq t \leq a\right\}\\
&&\\
\alpha_3^c&=&\{(a,0,t):0\leq t\leq c\}\\
&&\\
\alpha_4^c&=&\{(0,a,t): 0\leq t \leq c\}\\
&&\\
\alpha_5^c&=&\{(t,-t+a,c): 0\leq t \leq a\},
\end{array}$$
as illustrated in Figure \ref{fig1}.

\begin{figure}[h]
  \centering
  \includegraphics[height=5cm]{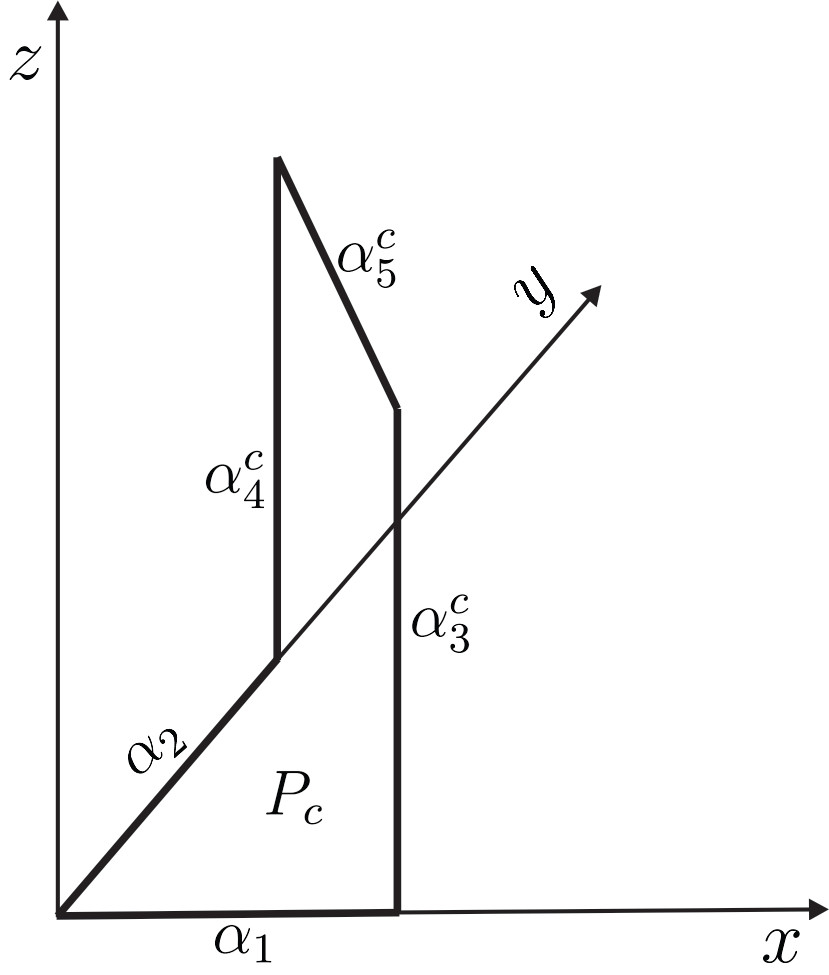}
\caption{Polygon $P_c.$}
\label{fig1}
\end{figure}

We will denote $\alpha_1^0=\{(t,0,0): 0\leq t< a\},$ $\alpha_2^0=\{(0,t,0): 0\leq t< a\},$ $\alpha_3=\{(a,0,t): t>0\}$ and $\alpha_4=\{(0,a,t): t>0\},$ hence $P_{\infty}=\alpha_1^0\cup\alpha_2^0\cup\alpha_3\cup\alpha_4\cup \{(a,0,0), (0,a,0)\}.$

Let $\Pi: \mathbb R^2 \rtimes_A \mathbb R \rightarrow \mathbb R^2 \rtimes_A \{0\}$ denote the projection $\Pi(x,y,z)=(x,y,0).$ The next proposition is proved in Lemma 1.2 in \cite{MMPR}, using the maximum principle and the fact that for every line $L\subset \mathbb R^2\rtimes_A \{0\},$ the vertical plane $\Pi^{-1}(L)$ is a minimal surface.

\begin{proposition}\label{prop-int}
Let $E$ be a compact convex disk in $\mathbb R^2 \rtimes_A \{0\}$ with boundary $C=\partial E$ and let $\Sigma$ be a compact minimal surface with boundary in $\Pi^{-1}(C).$ Then every point in $\mbox{\normalfont int}\Sigma$ is contained in $\mbox{\normalfont int}\Pi^{-1}(E).$
\end{proposition}

Observe that, for each $c>0,$ the polygon $P_c$ is transverse to the Killing field $X=\partial_x +\partial_y$ and each integral curve of $X$ intersects $P_c$ in at most one point. From now on, denote by $P$ the commom projection of every $P_c$ over $\mathbb R^2\rtimes_A \{0\},$ that is, $P=\Pi(P_c)=\Pi(P_d)$ for any $c,d \in \mathbb R,$ and denote by $E$ the disk in $\mathbb R^2\rtimes_A \{0\}$ with boundary $P.$ Let us denote by $\mathcal R$ the region $E\times\{z\geq 0\}.$ Using Theorem \ref{th-MMPR}, we conclude that $P_c$ is the boundary of a compact embedded disk $\Sigma_c$ of finite least area and the interior of $\Sigma_c$ is a smooth $\Pi$-graph over the interior of $E.$ 

Let $\Omega_c=\{(t,-t+a,s): 0\leq t\leq a; 0\leq s \leq c\}$.

\begin{proposition}
\label{prop-unique}
If $S$ is a compact minimal surface with boundary $P_c,$ then $S=\Sigma_c.$
\end{proposition}

\begin{proof}
By Proposition \ref{prop-int}, int$\Sigma_c,$ $\mbox{\normalfont int}S \subset \mbox{\normalfont int}\Pi^{-1}(E),$ then, in particular, int$\Sigma_c,$ $\mbox{\normalfont int}S \subset \mbox{\normalfont int}\{\varphi_t(p): t\in \mathbb R; p\in \Omega_c\},$ where $\varphi_t$ is the flow of the Killing field $X.$

As $S$ is compact, there exists $t$ such that $\varphi_t(\Sigma_c)\cap S=\emptyset.$ If $S\neq\Sigma_c,$ then there exists $t_0>0$ such that $\varphi_{t_0}(\Sigma_c)\cap S\neq\emptyset$ and for $t>t_0, \varphi_{t}(\Sigma_c)\cap S=\emptyset.$ Since for all $t \neq 0, \varphi_t(P_c)\cap S=\emptyset,$ then the point of intersection is an interior point and, by the maximum principle, $ \varphi_{t_0}(\Sigma_c)=S.$ But that is a contradiction, since $t_0\neq 0.$ Therefore, $S=\Sigma_c.$
\end{proof}

The next proposition is a classical result.

\begin{proposition}\label{prop-comp}
Let $N^3$ be a homogeneous three-manifold. Let $\Sigma_n$ be an oriented properly embedded minimal surface in $N.$ Suppose there exist $c>0$ such that for all $n,$  $|A_{\Sigma_n}|\leq c,$ and a sequence of points $\{p_n\}$ in $\Sigma_n$ such that $p_n\rightarrow p\in N.$ Then there exists a subsequence of $\Sigma_n$ that converges to a complete minimal surface $\Sigma$ with $p\in \Sigma.$ Here $A_{\Sigma_n}$ denotes the second fundamental form of $\Sigma_n.$
\end{proposition}

For each $n \in \mathbb N$, let $\Sigma_n$ be the solution to the Plateau problem with boundary $P_n.$  By Theorem \ref{th-MMPR} and Proposition \ref{prop-unique}, $\Sigma_n$ is stable and unique. We are interested in proving the existence of a subsequence of $\Sigma_n$ that converges to a complete minimal surface with boundary $P_{\infty}$. In order to do that, we will use a minimal annulus as a barrier (whose existence is guaranteed by the Douglas criterion (see \cite{J},  Theorem 2.1)) to show that there exist points $p_n \in \Sigma_n, \Pi(p_n)=q \in \mbox{int} E$ for all $n,$ which converge to a point $p\in \mathbb R^2\rtimes_A \mathbb R,$ and then we will use Proposition \ref{prop-comp}.  

Consider the parallelepiped with the faces $A, B, C, D, E$ and $F$, defined below.
$$\begin{array}{rcl}
A &=& \{(u, -\epsilon, v): \epsilon \leq u \leq a+\epsilon; c_0 \leq v\leq c_1\}\\
&&\\
B &=& \{(-\epsilon,u, v): \epsilon \leq u \leq a+\epsilon; c_0 \leq v\leq c_1\}\\
&&\\
C &=& \{(u,-u, v): -\epsilon \leq u \leq \epsilon; c_0 \leq v\leq c_1\}\\
&&\\
D &=& \{(u,-u+a, v): -\epsilon \leq u \leq a+\epsilon; c_0 \leq v\leq c_1\}\\
&&\\
E &=& \{(u,-u+v, c_0): -\epsilon \leq u \leq v+\epsilon; 0 \leq v\leq a\}\\
&&\\
F &=& \{(u,-u+v, c_1): -\epsilon \leq u \leq v+\epsilon; 0 \leq v\leq a\},
\end{array}$$
where $\epsilon$ is a positive constant that we will choose later. Observe that each one of these faces is the least area minimal surface with its boundary. Let us analyse the area of each face.

1. In the plane $\{y = \mbox{constant}\}$ the induced metric is given by $g(x,z)=(a_{11}^2(z)+a_{21}^2(z))dx^2 +dz^2.$ Hence,
$$\begin{array}{rcl}
\mbox{area} \ A&=& \displaystyle\int_{c_0}^{c_1}\int_{\epsilon}^{a+\epsilon}\sqrt{a_{11}^2(z)+a_{21}^2(z)}{\rm d}x{\rm d}z \\
&&\\
&=& a\displaystyle \int_{c_0}^{c_1}\sqrt{a_{11}^2(z)+a_{21}^2(z)}{\rm d}z.\\
\end{array}$$

2. In the plane $\{x = \mbox{constant}\}$ the induced metric is given by $g(y,z)=(a_{11}^2(z)+a_{12}^2(z))dy^2 +dz^2.$ Hence,
$$\begin{array}{rcl}
\mbox{area} \ B&=&  \displaystyle\int_{c_0}^{c_1}\int_{\epsilon}^{a+\epsilon}\sqrt{a_{11}^2(z)+a_{12}^2(z)}{\rm d}x{\rm d}z \\
&&\\
&=& a\displaystyle \int_{c_0}^{c_1}\sqrt{a_{11}^2(z)+a_{12}^2(z)}{\rm d}z.\\
\end{array}$$

3. The face $C$ is contained in the plane parameterized by $\phi(u,v)=(u,-u,v)$ and the face $D$ is contained in the plane parameterized by $\psi(u,v)=(u,-u+a,v)$. We have $\psi_u=\phi_u=\partial_x -\partial_y, \psi_v=\phi_v=\partial_z.$ Then, $|\psi_u\wedge\psi_v|=|\phi_u\wedge\phi_v|=\sqrt{(a_{11}(z)+a_{12}(z))^2 + (a_{11}(z)+a_{21}(z))^2}.$ Hence,
$$\begin{array}{rcl}
\mbox{area} \ C&=&\displaystyle\int_{c_0}^{c_1}\int_{-\epsilon}^{+\epsilon}\sqrt{(a_{11}(z)+a_{12}(z))^2 + (a_{11}(z)+a_{21}(z))^2}{\rm d}u{\rm d}v\\
&&\\
& =& \displaystyle 2\epsilon \int_{c_0}^{c_1}\sqrt{(a_{11}(z)+a_{12}(z))^2 + (a_{11}(z)+a_{21}(z)^2}{\rm  d}z,\\
\end{array}$$
$$\begin{array}{rcl}
\mbox{area} \ D&=&\displaystyle\int_{c_0}^{c_1}\int_{-\epsilon}^{a+\epsilon}\sqrt{(a_{11}(z)+a_{12}(z))^2 + (a_{11}(z)+a_{21}(z))^2}{\rm d}u{\rm d}v\\
&&\\
&=& \displaystyle(a+2\epsilon) \int_{c_0}^{c_1}\sqrt{(a_{11}(z)+a_{12}(z))^2 + (a_{11}(z)+a_{21}(z))^2}{\rm d}z.
\end{array}$$

4. As the plane $\{z= \mbox{constant}\}$ is flat, then the induced metric is the Euclidean metric. Hence,
$$
\mbox{area} \ E=\mbox{area} \ F= \int_0^a\int_{-\epsilon }^{v+\epsilon}{\rm d}u{\rm d}v=\frac{a(a+4\epsilon)}{2}.
$$

Therefore,
$$\mbox{area} \ C + \mbox{area} \ D+ \mbox{area} \ E +\mbox{area}\ F<\mbox{area}\ A+\mbox{area}\ B$$
se, e somente se,
$$\begin{array}{rcl}(a+4\epsilon)\displaystyle\left[a+\int_{c_0}^{c_1}\sqrt{(a_{11}+a_{12})^2 + (a_{11}+a_{21})^2}{\rm d}z\right] &<&a\displaystyle\int_{c_0}^{c_1}\sqrt{a_{11}^2+a_{21}^2}{\rm d}z\\
&+&a\displaystyle\int_{c_0}^{c_1}\sqrt{a_{11}^2+a_{12}^2}{\rm d}z
\end{array}$$
se, e somente se,
\begin{equation}
\epsilon<\frac{a}{4}\frac{\displaystyle\int_{c_0}^{c_1}\sqrt{a_{11}^2(z)+a_{21}^2(z)}+\sqrt{a_{11}^2(z)+a_{21}^2(z)}{\rm d}z}{a+\displaystyle\int_{c_0}^{c_1}\sqrt{(a_{11}(z)+a_{12}(z))^2 + (a_{11}(z)+a_{21}(z))^2}{\rm d}z}-\frac{a}{4}.
\label{eqdefe}
\end{equation}

As we chose $a$ satisfying $(\ref{eqdefa}),$ the factor in the right hand side of (\ref{eqdefe}) is a positive number, then we can choose $\epsilon>0$ such that the Douglas criterion is satisfied \cite{J}. Hence we obtain a minimal annulus $ \mathcal A$ with boundary $\partial A\cup \partial B$ such that its projection $\Pi(\mathcal A )$ contains points of int$E,$ where $E$ is the disk in $\mathbb R\rtimes_A \{0\}$ with boundary $P.$ (See Figure \ref{fig-annulus}).

\begin{figure}[h]
  \centering
  \includegraphics[height=7cm]{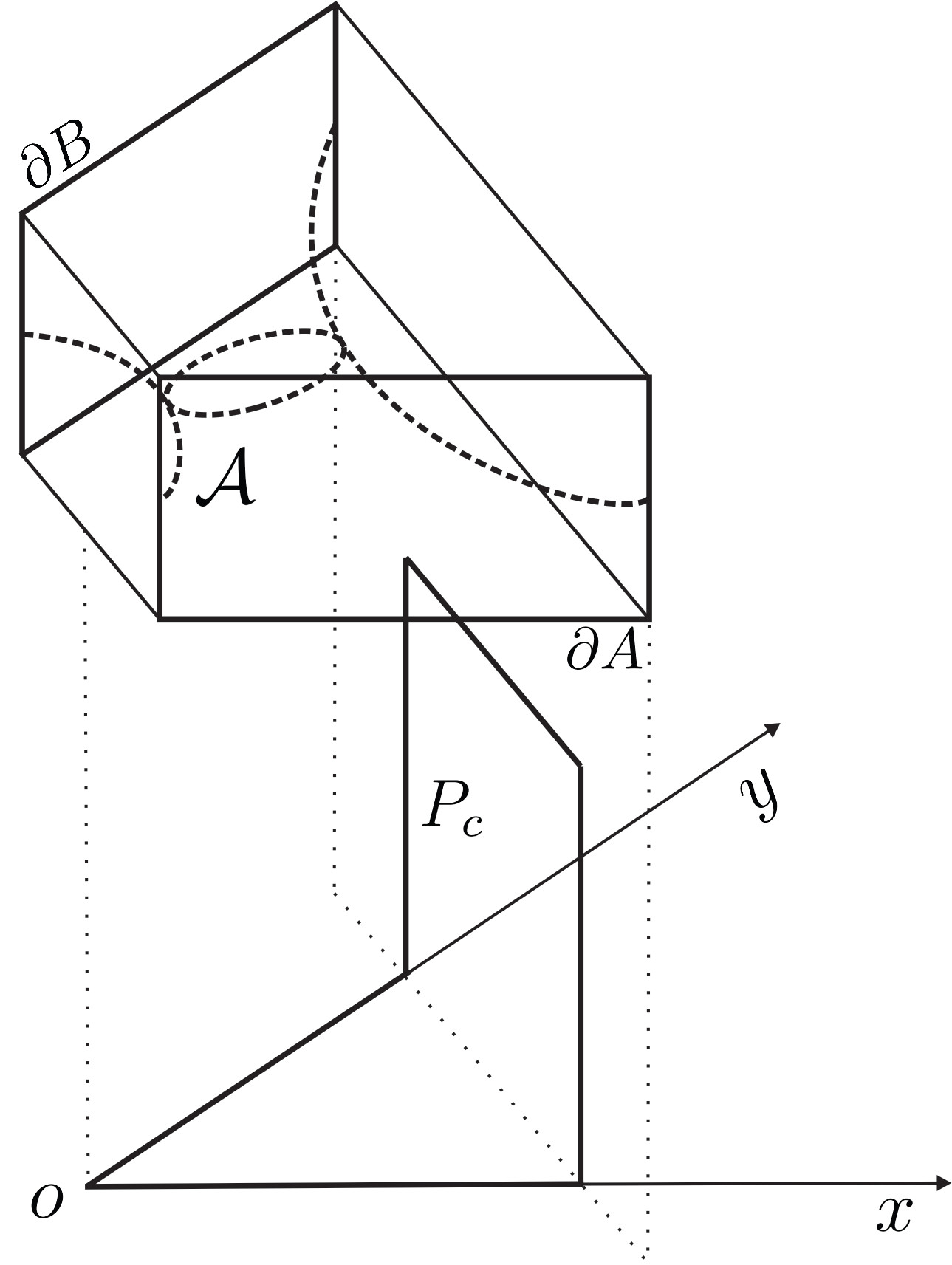}
\caption{Annulus $\mathcal A$.}
\label{fig-annulus}
\end{figure}

As $\mathbb R^2\rtimes_A \{z\}$ is a minimal surface, the maximum principle implies that, for each $c,$ $\Sigma_c$ is contained in the slab bounded by the planes $\{z=0\}$ and $\{z=c\}.$ Then for $c<c_0,$ $\Sigma_c \cap \mathcal A=\emptyset.$ As $\Sigma_c$ is unique, $\Sigma_c$ varies continuously with $c,$ and when $c$ increases the boundary $\partial\Sigma_c=P_c$ does not touch $\partial \mathcal A.$ Therefore, using the maximum principle, $\Sigma_c\cap\mathcal A=\emptyset$ for all $c,$ and $\Sigma_c$ is under the annulus $\mathcal A,$ which means that over any vertical line that intersects $\mathcal A$ and $\Sigma_c,$ the points of $\Sigma_c$ are under the points of $\mathcal A.$ 

Consider $\varphi_{t}$ the flow of the Killing field $X=\partial_x+\partial_y.$ Observe that $\{\varphi_t(\mathcal A)\}_{t<0}$ forms a barrier for all points $p_n\in\Sigma_n$ such that $\Pi(p_n)$ is contained in a neighborhood $\mathcal U \subset E$ of the origin $o=(0,0,0).$ Moreover, for any $c_2<c_3$ we can use the flow $\varphi_t$ of the Killing field $X$ and the maximum principle to conclude that $\Sigma_{c_2}$ is under $\Sigma_{c_3}$ in the same sense as before.

As, by Theorem $\ref{th-MMPR}$, each $\Sigma_n$ is a vertical graph in the interior, then $\Sigma_n\cap \Pi^{-1}(q)$ is only one point $p_n,$ for every point $q \in$ int$E.$ Moreover, by the previous paragraph, the sequence $p_n=\Sigma_n\cap \Pi^{-1}(q)$ is monotone. Then, since we have a barrier, the sequence $\{p_n=\Sigma_n\cap\Pi^{-1}(q)\}$ converges to a point $p\in \Pi^{-1}(q),$ for all $q\in \mathcal U.$

In order to understand the convergence of the surfaces $\Sigma_n$ we need to observe some properties of these surfaces.

First, notice that, rotation by angle $\pi$ around $\alpha_3,$ that we will denote by $R_{\alpha_3},$ is an isometry. By the Schwarz reflection, we obtain a minimal surface $\widetilde{\Sigma}_n=\Sigma_n\cup R_{\alpha_3}(\Sigma_n)$ that has int$\alpha_3$ in its interior. Note that the boundary of $\widetilde{\Sigma}_n$ is transverse to the Killing field $X=\partial_x+\partial_y,$ and if $\varphi_t$ denotes the flow of $X,$ we have that $\varphi_t(\partial\widetilde{\Sigma}_n)\cap \widetilde{\Sigma}_n=\emptyset$ for all $t\neq0,$  hence, using the same arguments of the proof of Proposition \ref{prop-unique}, we can show that the minimal surface $\widetilde{\Sigma}_n$ is the unique minimal surface with its boundary. In particular, it is area-minimizing, and then it is stable. Hence, by Main Theorem in \cite{RST}, we have uniform curvature estimates for points far from the boundary of $\widetilde{\Sigma}_n.$ In particular, we get uniform curvature estimates for $\Sigma_n$ in a neighborhood of $\alpha_3.$ Analogously, we have uniform curvature estimates for $\Sigma_n$ in a neighborhood of $\alpha_4.$ 

Hence, for every compact contained in $\{z>0\}\cap \mathcal R,$ there exists a subsequence of $\Sigma_n$ that converges to a minimal surface. Taking an exhaustion by compact sets and using a diagonal process, we conclude that there exists a subsequence of $\Sigma_n$ that converges to a minimal surface $\Sigma$ that has $\alpha_3\cup\alpha_4$ in its boundary. From now on, we will use the notation $\Sigma_n$ for this subsequence.

It remains to prove that in fact $\Sigma$ is a minimal surface with boundary $P_\infty.$ In order to do that, we will use the fact that the interior of each $\Sigma_n$ is a vertical graph over the interior of $E$. Let us denote by $u_n$ the function defined in int$E$ such that $\Sigma_n=\mbox{Graph}(u_n).$ We already know that $u_{n-1}< u_n$ in int$E$ for all $n.$ 

\begin{claim}
There are uniform gradient estimates for $\{u_n\}$ for points in $\alpha_1^0\cup\alpha_2^0.$
\label{claim1}
\end{claim}

\begin{proof} For $x_0<0$ and $\delta>0$ consider the vertical strip bounded by $\beta_1=\{(x_0,y, c_1): -\delta\leq y\leq 0\},$ $\beta_2=\{(x_0, t,-\frac{c_1}{a}t+c_1): 0\leq t\leq a\},$ $\beta_3=\{(x_0,t-\delta, -\frac{c_1}{a}t+c_1): 0\leq t\leq a\}$ and $\beta_4=\{(x_0,y,0): a-\delta\leq y\leq a\}.$ This is a minimal surface transversal to the Killing field $\partial_x,$ hence any small perturbation of its boundary gives a minimal surface with that perturbed boundary. Thus, if we consider a small perturbation of the boundary of this vertical strip by just perturbing a little bit $\beta_1$ by a curve contained in $\{x\geq x_0\}$ joining the points $(x_0, -\delta, c_1)$ and $(x_0, 0, c_1),$ we will get a minimal surface $S$ with this perturbed boundary. This minimal surface $S$ will have the property that the tangent planes at the interior of $\beta_4$ are not vertical, by the maximum principle with boundary.

Applying translations along the $x$-axis and $y$-axis, we can use the translates of $S$ to show that $\Sigma_n$ is under $S$ in a neighborhood of $\alpha_2^0,$ and then we have uniform gradient estimates for points in $\alpha_2^0.$ Analogously, constructing similar barriers, we can prove that we have uniform gradient estimates in a neighborhood of $\alpha_1^0.$ 
\end{proof}

Observe that besides the gradient estimates, the translates of the minimal surface $S$ form a barrier for points in a neighborhood of $\alpha_1^0\cup \alpha_2^0.$

We have that $\Sigma_n$ is a monotone increasing sequence of minimal graphs with uniform gradient estimates in $\alpha_1^0\cup \alpha_2^0,$ and it is a bounded graph for points in a neighborhood $\mathcal U$ of the origin (because of the barrier given by the annulus $\mathcal A$). Therefore, there exists a subsequence of $\Sigma_n$ that converges to a minimal surface ${\widetilde \Sigma}$ with $\alpha_1^0\cup \alpha_2^0$ in its boundary. As we already know that $\Sigma_n$ converges to the minimal surface $\Sigma,$ we conclude that in fact $\Sigma=\widetilde{\Sigma},$ and then $\Sigma$ is a minimal surface with $\alpha_1^0\cup\alpha_2^0\cup\alpha_3\cup\alpha_4$ in its boundary. Notice that we can assume that $\Sigma$ has $P_\infty$ as its boundary, with $\Sigma$ being of class $C^1$ up to $P_\infty\setminus \{(a,0,0),(0,a,0)\}$ and continuous up to $P_\infty.$ 

Now considering the rotation by angle $\pi$ around $\alpha_1$ of $\Sigma,$ we obtain the surface illustrated in Figure \ref{fig-reflected}.
\begin{figure}[h]
  \centering
  \includegraphics[height=6cm]{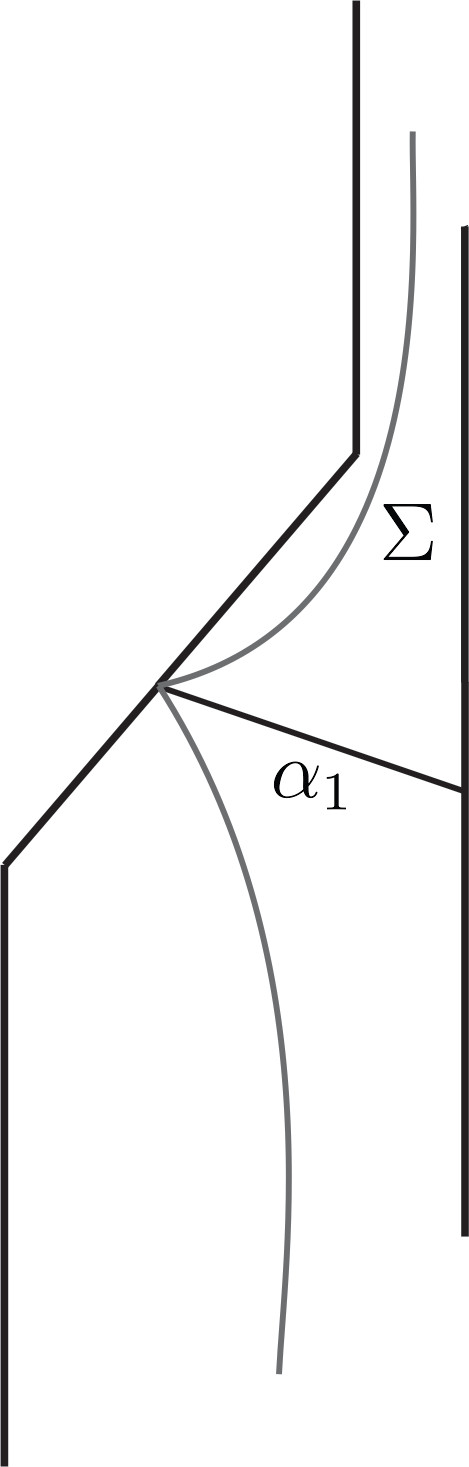}
  \caption{Rotation by angle $\pi$ around $\alpha_1$ of $\Sigma.$}
\label{fig-reflected}
\end{figure}

Continuing to rotate by angle $\pi$ around the $y$-axis, the resulting surface will be a minimal surface with four vertical lines as its boundary: $\{(a,0,t): t\in \mathbb R\}, \{(0,a,t): t\in \mathbb R\}, \{(-a,0,t): t\in \mathbb R\}, \{(0,-a,t): t\in \mathbb R\}.$ 

Now we can use the rotations by angle $\pi$ around the vertical lines to get a complete minimal surface that is analogous to the doubly periodic minimal Scherk surface in $\mathbb R^3.$ It is invariant by two translations that commute and it is a four punctured sphere in the quotient of $\mathbb R^2\rtimes_A \mathbb R$ by the group of isometries generated by the two translations.

\begin{theorem}
In any semidirect product $\mathbb R^2\rtimes_A \mathbb R,$ where $A=\left(
\begin{array}{ccc}
	0 & b\\
	c&0
\end{array}\right),$ there exists a periodic minimal surface similar to the doubly periodic Scherk minimal surface in $\mathbb R^3.$
\end{theorem}

\section{A singly periodic Scherk minimal surface}

Throughout this section, we consider the semidirect product $\mathbb R^2\rtimes_A \mathbb R$ with the canonical left invariant metric $\left\langle, \right\rangle,$ where $A=\left(\begin{array}{cc} 0&b\\c&0\end{array}\right).$ In this space, we construct a complete minimal surface similar to the singly periodic Scherk minimal surface in $\mathbb R^3$.

Fix $c_0>0$ and take $0<\epsilon < a$ sufficiently smalls so that
$$
a+2\epsilon < \int_{0}^{c_0}{\sqrt{a_{11}^2(z)+a_{21}^2(z)}}dz.
$$

For each $c>0,$ consider the polygon $P_c$ in $\mathbb {R}^2 \rtimes_A \mathbb{R}$ with the six sides defined below.
$$
\begin{array}{rcl}
\alpha_1^c &=& \left\{(t,0,0):0\leq t \leq c\right\}\\
&&\\
\alpha_2^c&=&\left\{(c,t,0): 0\leq t \leq a\right\}\\
&&\\
\alpha_3^c&=&\{(t,a,0): 0\leq t \leq c\} \\
&&\\
\alpha_4^c&=&\{(0,a,t): 0\leq t \leq c\}\\
&&\\
\alpha_5^c&=&\{(0,t,c):  0\leq t \leq a\}\\
&&\\
\alpha_6^c&=&\{(0,0,t): 0\leq t\leq c \},
\end{array}$$
and for each $\delta>0$ with $\delta<a/2,$ consider the polygon $P_c^{\delta}$ with the following six sides.
$$\begin{array}{rcl}
\alpha_1^{ \delta, c} &=& \left\{(t,\frac{\delta}{c}t,0):0\leq t \leq c\right\}\\
&&\\
\alpha_2^{ \delta,c}&=&\{(c,t,0): \delta \leq t \leq a-\delta\}\\
&&\\
\alpha_3^{\delta,c}&=&\left\{(t,\frac{ac-\delta t}{c},0):  0\leq t \leq c\right\},
\end{array}$$
$\alpha_4^c, \alpha_5^c, \alpha_6^c,$  as illustrated in Figure \ref{fig4}.

\begin{figure}[h!]
  \centering
  \includegraphics[width=11.5cm]{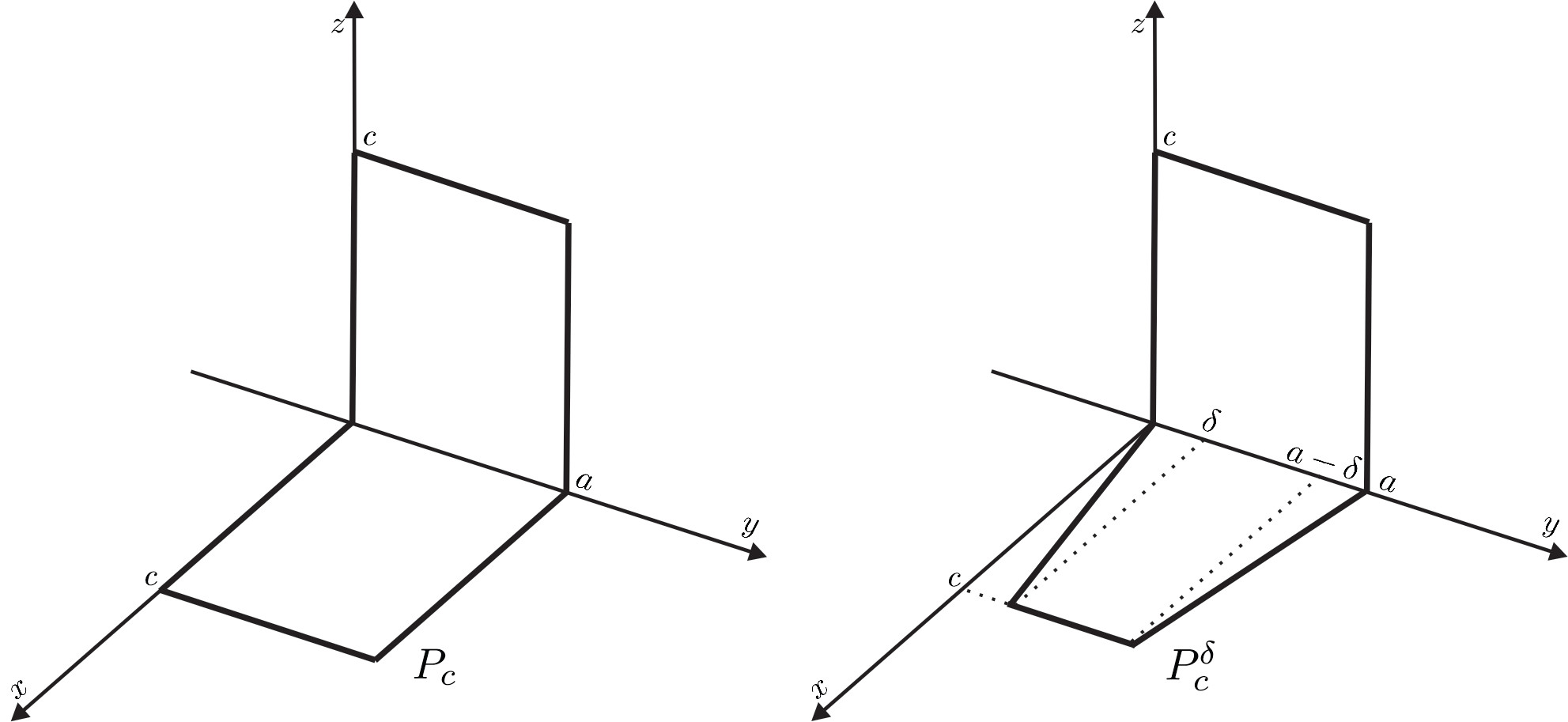}
\caption{Polygons $P_c$ and $P_c^{\delta}.$}
\label{fig4}
\end{figure} 

Denote by $\Omega(\delta, c)$ the region in $\mathbb R^2\rtimes_A \{0\}$ bounded by $\alpha_1^{ \delta,c}, \alpha_2^{ \delta,c}, \alpha_3^{ \delta,c}$ and the segment $\{(0,t,0): 0\leq t\leq a\}.$ For each $c$ and $\delta,$ we have compact minimal surfaces $\Sigma_c$ and $\Sigma_c^{\delta}$ with boundary $P_c$ and $P_c^{\delta},$ respectively, which are solutions to the Plateau problem. By Theorem \ref{th-MMPR}, we know that $\Sigma_c$ and $\Sigma_c^{\delta}$ are stable and smooth $\Pi$-graphs over the interior of $\Omega(0,c), \Omega(\delta,c),$ respectively. We will show that $\Sigma_c$ is the unique compact minimal surface with boundary $P_c.$

Fix $c.$ For each $0< \delta <a/2,$ $P_c^{\delta}$ is a polygon transverse to the Killing field $\partial_x$ and each integral curve of $\partial_x$ intersects $P_c^{\delta}$ in at most one point. Thus we can prove, as we did in Proposition \ref{prop-unique}, that $\Sigma_c^{\delta}$ is the unique compact minimal surface with boundary $P_c^{\delta}.$

Denote by $u_c^{\delta}, v_c$ the functions defined in the interior of $\Omega(\delta,c), \Omega(0,c),$ whose $\Pi$-graphs are $\Sigma_c^{\delta}, \Sigma_c,$ respectively. Then, as $\partial_x$ is a Killing field and each $P_c^{\delta}$ is transversal to $\partial_x,$ we can use the flow of $\partial_x$ and the maximum principle to prove that for $\delta'<\delta$ we have $0\leq u_c^{\delta} \leq u_c^{\delta'} \leq v_c$ in int$\Omega(\delta,c),$ hence $v_c$ is a barrier for our sequence $u_c^{\delta}.$ Because of the monotonicity and the barrier, the family $u_c^{\delta}$ converges to a function $u_c$ defined in int$\Omega(0,c)$ whose graph is a compact minimal surface with boundary $P_c,$ and we still have $u_c\leq v_c$ on $\Omega(0,c).$

Now we will find another compact minimal surface with boundary $P_c,$ whose interior is the graph of a function $w_c$ defined in int$\Omega(0,c)$ such that $v_c\leq w_c$ and we will show that $u_c=w_c.$ 
 In order to do that, for each $0<\delta <a/2,$ consider the polygon $\widetilde{P}_c^{\delta}$ with the six sides defined below. (See Figure \ref{tilda}).
$$\begin{array}{rcl}
\widetilde{\alpha}_1^{\delta,c}&=&\{(t,\frac{\delta t -\delta c}{c},0): 0\leq t\leq c\}\\
&&\\
\alpha_2^c&=&\{(c,t,0): 0\leq t\leq a\}\\
&&\\
\widetilde{\alpha}_3^{\delta,c}&=&\{(t,\frac{(a+\delta)c-\delta t}{c},0): 0\leq t\leq c\}\\
&&\\
\widetilde{\alpha}_4^{\delta,c}&=&\{(0,a+\delta,t): 0\leq t\leq c\}\\
&&\\
\widetilde{\alpha}_5^{\delta,c}&=&\{(0,t,c):-\delta\leq t \leq a+\delta\}\\
&&\\
\widetilde{\alpha}_6^{\delta,c}&=&\{(0,-\delta, t): 0\leq t\leq c\}.
\end{array}$$

\begin{figure}[h!]
  \centering
  \includegraphics[width=11.5cm]{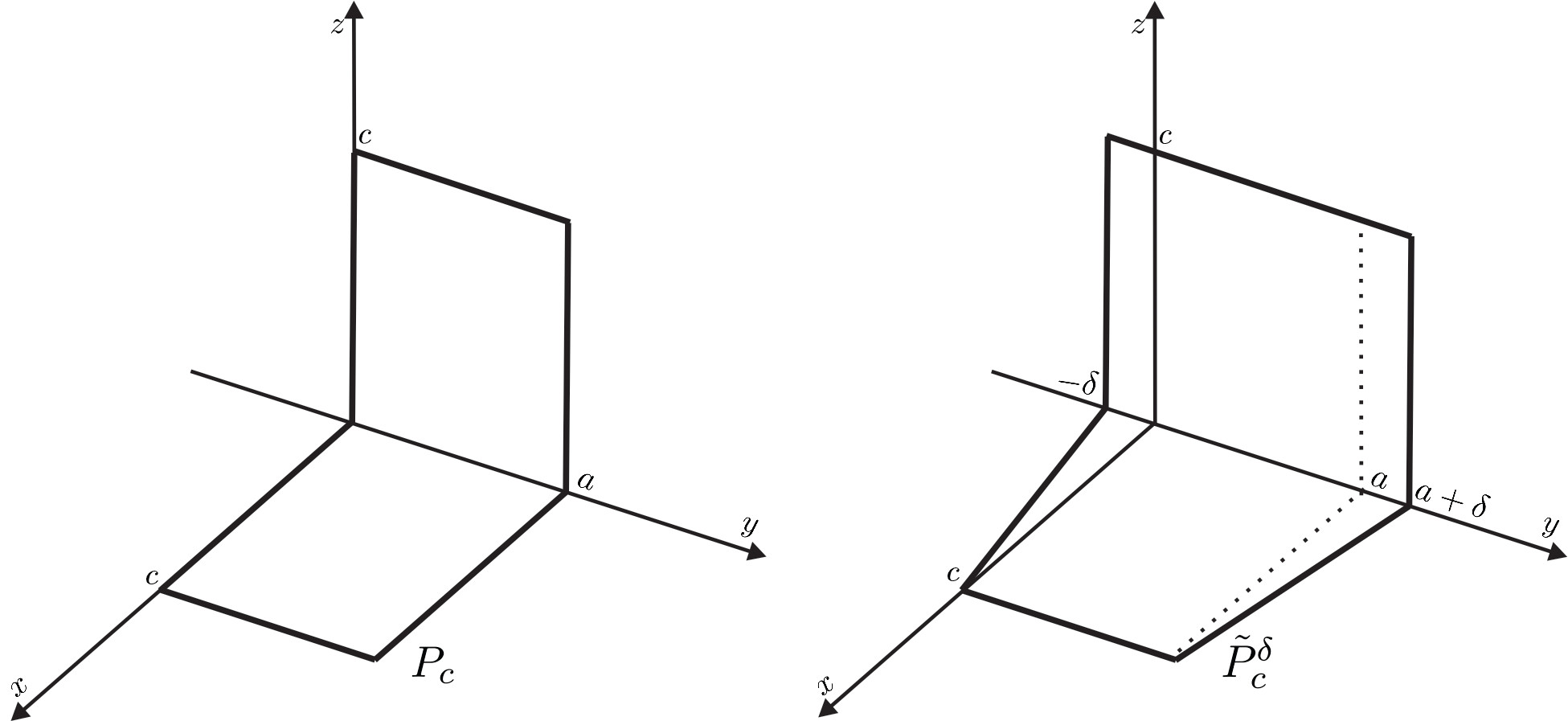}
\caption{Polygons $P_c$ and $\widetilde{P}_c^{\delta}.$}
\label{tilda}
\end{figure}

Denote by $\widetilde{\Omega}(\delta,c)$ the region in $\mathbb R^2\rtimes_A \{0\}$ bounded by $\widetilde{\alpha}_1^{\delta,c},$ \ $\alpha_2^c,\  \widetilde{\alpha}_3^{\delta,c}$ and the segment $\{(0,t,0): -\delta\leq t \leq a+\delta\}.$ For each $\delta,$ we have a compact minimal disk $\widetilde{\Sigma}_c^{\delta}$ with boundary $\widetilde{P}_c^{\delta}$ and $\widetilde{\Sigma}_c^{\delta}$ is a smooth $\Pi$-graph over the interior of $\widetilde{\Omega}(\delta,c).$ As $\widetilde{P}_c^{\delta}$ is transversal to the Killing field $\partial_x,$ we can prove that $\widetilde{\Sigma}_c^{\delta}$ is the unique compact minimal surface with boundary $\widetilde{P}_c^{\delta}.$ 

Denote by $w_c^{\delta}$ the function defined in int$\widetilde{\Omega}(\delta,c)$ whose graph is $\widetilde{\Sigma}_c^{\delta}.$ Using the flow of $\partial_x$ and the maximum principle, we can prove that for $\delta'<\delta$ we have $w_c^{\delta'}\leq w_c^{\delta}$ in int$\widetilde{\Omega}(\delta',c)$ and for all $\delta,$ $v_c\leq w_c^{\delta}$ in int$\Omega(0,c).$ Because of the monotonicity and the barrier, the family $w_c^{\delta}$ converges to a function $w_c$ defined in int$\widetilde{\Omega}(0,c)=\mbox{int}\Omega(0,c)$ whose graph is a compact minimal surface with boundary $P_c,$ and we still have $v_c\leq w_c$ in int$\Omega(0,c).$

Let us call $\Sigma_1, \Sigma_2$ the graphs of $u_c,w_c,$ respectively. We will now prove that $\Sigma_1=\Sigma_2.$ Denote by $\nu_i$ the conormal to $\Sigma_i$ along $P_c,$ $i=1,2.$ (See Figure \ref{conormal}).

\begin{figure}[h!]
  \centering
  \includegraphics[height=5.8cm]{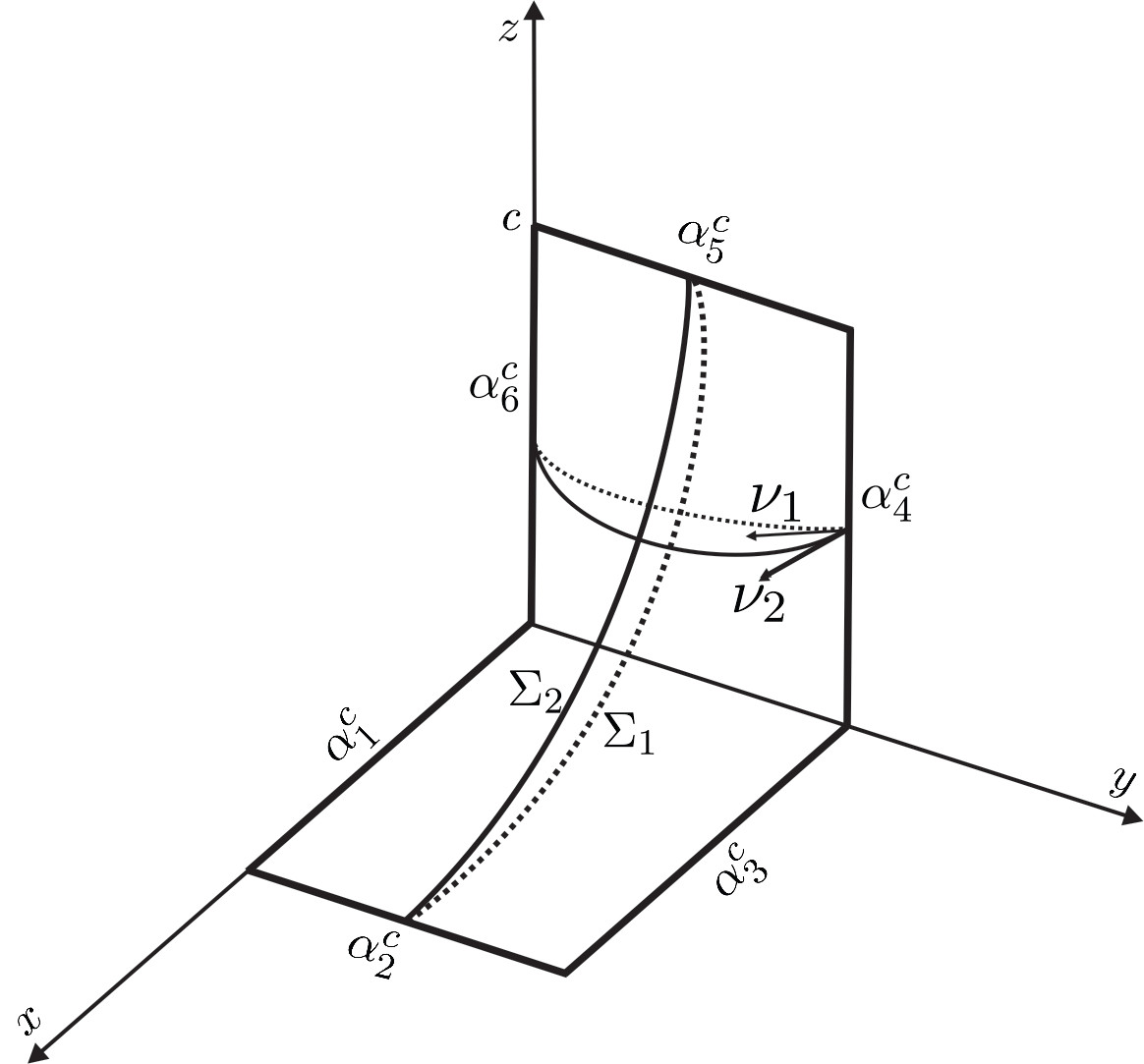}
\caption{$\Sigma_1$ and $\Sigma_2.$}
\label{conormal}
\end{figure}

Suppose that $u_c\neq w_c,$ then in fact we have $u_c<w_c$ in int$\Omega(0,c).$ As $\partial_x$ is tangent to $\alpha_1^c$ and $\alpha_3^c,$ then $\left\langle \nu_i, \partial_x\right\rangle=0,$ $i=1,2,$ in $\alpha_1^c$ and $\alpha_3^c.$ In the other sides of $P_c$ we have $\left\langle \nu_1, \partial_x\right\rangle <\left\langle \nu_2, \partial_x\right\rangle.$ Therefore,
$$
\int_{P_c}\left\langle \nu_1, \partial_x\right\rangle < \int_{P_c}\left\langle \nu_2, \partial_x\right\rangle.
$$
But, using the Flux Formula for $\Sigma_1$ and $\Sigma_2$ with respect to the Killing field $\partial_x,$ we have
$$
\int_{P_c}\left\langle \nu_1, \partial_x\right\rangle =0= \int_{P_c}\left\langle \nu_2, \partial_x\right\rangle.
$$
Then, $u_c=w_c$ and therefore, $\Sigma_c=\Sigma_1=\Sigma_2.$ In particular, $\Sigma_c$ is the unique compact minimal surface with boundary $P_c.$

Denote by $\Omega(\infty)$ the infinite strip $\{(x,y,0): x\geq 0, 0\leq y\leq a\},$ and by $\mathcal R$ the region $\{(x,y,z):x\geq0, 0\leq y\leq a, z\geq 0\}.$ Moreover, denote $\alpha_1=\{(x,0,0):x>0\},$ $\alpha_3=\{(x,a,0):x>0\},$ $\alpha_4=\{(0,a,z):z>0\}$ and $\alpha_6=\{(0,0,z):z>0\},$ hence $P_{\infty}=\alpha_1\cup\alpha_3\cup\alpha_4\cup\alpha_6\cup\{(0,0,0),(0,a,0)\}.$ 

For each $n\in \mathbb N,$ let $\Sigma_n$ be the unique compact minimal surface with boundary $P_n.$ We are interested in proving the existence of a subsequence of $\Sigma_n$ that converges to a complete minimal surface with boundary $P_{\infty}$. Using the existence of a minimal annulus, guaranteed by the Douglas criterion, we will show that there exist points $p_n\in \Sigma_n,\ \Pi(p_n)=q \in \mbox{int}\ \Omega(\infty)$ for all $n,$ which converge to a point $p\in \mathbb R^2\rtimes_A \mathbb R,$ and then we will use Proposition \ref{prop-comp}.

Consider the parallelepiped with faces $A, B, C, D, E$ and $F,$ defined below.
$$\begin{array}{rcl}
A &=& \{(u, -\epsilon, v): \epsilon \leq u \leq d; 0 \leq v\leq c_0\}\\
&&\\
B &=& \{(u, a+\epsilon, v): \epsilon \leq u \leq d; 0 \leq v\leq c_0\}\\
&&\\
C &=& \{(u,v, 0): \epsilon \leq u \leq d; -\epsilon \leq v\leq a+\epsilon \}\\
&&\\
D &=& \{(u,v, c_0): \epsilon \leq u \leq d; -\epsilon \leq v\leq a+\epsilon\}\\
&&\\
E &=& \{(\epsilon,u,v): -\epsilon \leq u \leq a+\epsilon; 0 \leq v\leq c_0\}\\
&&\\
F &=& \{(d,u, v): -\epsilon \leq u \leq a+\epsilon; 0 \leq v\leq c_0\},
\end{array}$$
where $d>\epsilon$ is a constant that we will choose later.

As we did in the last section, we can calculate the area of each one of these faces and we obtain:
$$\begin{array}{rcl}
\mbox{area}\ A = \mbox{area}\ B&=& (d-\epsilon)\displaystyle\int_0^{c_0}\sqrt{a_{11}^2(z)+a_{21}^2(z)}{\rm d}z, \\
&&\\
\mbox{area}\ C=\mbox{area}\ D&=&(d-\epsilon)(a+2\epsilon),\\
&&\\
\mbox{area}\ E=\mbox{area}\ F&=&(a+2\epsilon)\displaystyle\int_0^{c_0}\sqrt{a_{11}^2(z)+a_{12}^2(z)}{\rm d}z.
\end{array}$$

Hence,
$$
\mbox{area}\ C + \mbox{area}\ D + \mbox{area}\ E+\mbox{area}\ F< \mbox{area}\ A+\mbox{area}\ B
$$
se, e somente se,
$$(d-\epsilon)(a+2\epsilon) + (a+2\epsilon)\displaystyle\int_0^{c_0}\sqrt{a_{11}^2+a_{12}^2}{\rm d}z < (d-\epsilon)\displaystyle\int_0^{c_0}\sqrt{a_{11}^2+a_{21}^2}{\rm d}z$$
se, e somente se,
$$
(d-\epsilon)\displaystyle \left[(a+2\epsilon)- \int_0^{c_0}\sqrt{a_{11}^2+a_{21}^2}{\rm d}z\right]< - (a+2\epsilon)\int_0^{c_0}\sqrt{a_{11}^2+a_{12}^2}{\rm d}z
$$
se, e somente se,
$$
d> \epsilon - \frac{(a+2\epsilon)\displaystyle\int_0^{c_0}\sqrt{a_{11}^2(z)+a_{12}^2(z)}{\rm d}z}{(a+2\epsilon)- \displaystyle\int_0^{c_0}\sqrt{a_{11}^2(z)+a_{21}^2(z)}{\rm d}z}.
$$

As we chose $a+2\epsilon < \displaystyle\int_0^{c_0}\sqrt{a_{11}^2(z)+a_{21}^2(z)}{\rm d}z,$ we can choose $d>\epsilon$ so  that the Douglas criterion is satisfied \cite{J}. Thus, there exists a minimal annulus $\mathcal A$ with boundary $\partial A \cup \partial B$ such that its projection $\Pi(\mathcal A)$ contains points of int$\Omega(\infty).$ (See Figure \ref{fig-annulus2}).

\begin{figure}[h!]
  \centering
  \includegraphics[height=6cm]{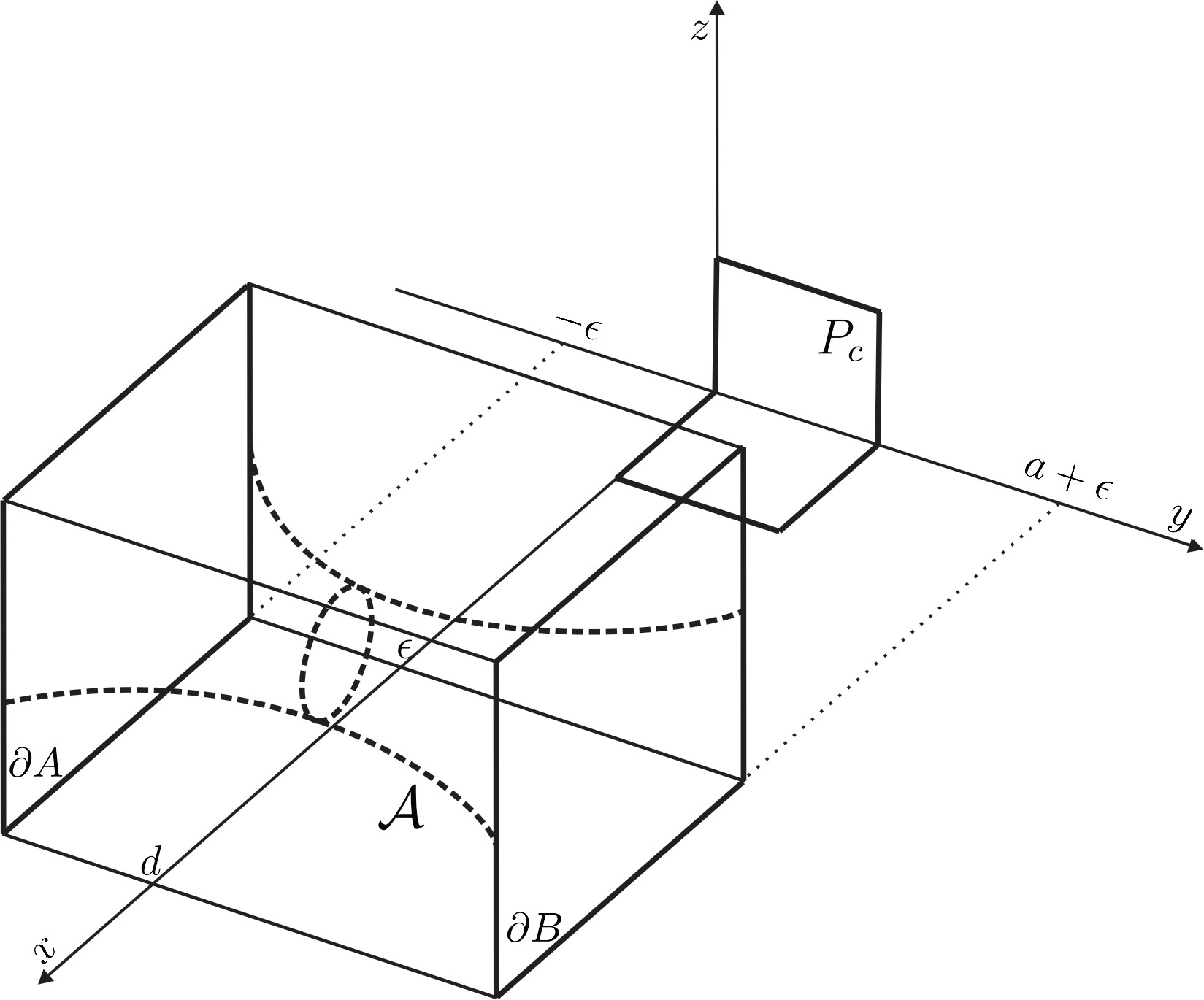}
  \caption{Annulus $\mathcal A$.}
  \label{fig-annulus2}
\end{figure}

We know that, for each $c<\epsilon,$ $\Sigma_c\cap\mathcal A=\emptyset.$ When $c$ increases $P_c$ does not intersect $\partial \mathcal A,$ then, using the maximum principle, $\Sigma_c\cap\mathcal A=\emptyset$ for all $c,$ and $\Sigma_c$ is under the annulus $\mathcal A.$ Thus, there exists a point $q\in \mbox{int} \Omega(\infty)$ such that $p_n=\Sigma_n\cap\Pi^{-1}(q)$ has a subsequence that converges to a point $p\in \Pi^{-1}(q).$ Observe that applying the flow of the Killing field $\partial_x$ to the annulus $\mathcal A$ we can conclude that, in the region $\{x\geq d\},$ the surfaces $\Sigma_n$ are bounded above by, for example, the plane $\{z=c_0\}.$ 

In order to understand the convergence of the surfaces $\Sigma_n$ we need to prove some properties of these surfaces.

\begin{claim}
The surfaces $\Sigma_n$ are transversal to the Killing field $\partial_x$ in the interior. 
\end{claim}

\begin{proof} Fix $n.$ Suppose that at some point $p\in\mbox{int}\Sigma_n$ the tangent plane $T_p\Sigma_n$ contains the vector $\partial_x.$ As the planes that contain the direction $\partial_x$ are minimal surfaces, we have that $\Sigma_n$ and $T_p\Sigma_n$ are minimal surfaces tangent at $p,$ and then the intersection between them is formed by $2k$ curves, $k\geq 1,$ passing through $p$ making equal angles at $p.$ By the shape of $P_n$ (the boundary of $\Sigma_n$), we know that $T_p\Sigma_n$ intersects $P_n$ either in only two points or in one point and a segment of straight line ($\alpha_1^n$ or $\alpha_3^n$). Therefore, we will have necessarily a closed curve contained in the intersection. As $\Sigma_n$ is simply connected this curve bounds a disk in $\Sigma_n,$ but as the parallel planes to $T_p\Sigma_n$ are minimal surfaces, we can use the maximum principle to prove that this disk is contained in the plane $T_p\Sigma_n$ and then they coincide, which is impossible. Thus, the vector $\partial_x$ is transversal to $\Sigma$ at points $p\in\mbox{int}\Sigma_n.$
\end{proof}
 
Observe that, besides the interior points, the surfaces $\Sigma_n$ are also transversal to $\partial_x$ at the points in $\alpha_4$ and $\alpha_6,$ by the maximum principle with boundary. Thus rotation by angle $\pi$ around $\alpha_4$ (respectively $\alpha_6$) gives a minimal surface which is also transversal to the Killing field $\partial_x$ in the interior, extends the surface $\Sigma_n$ and has $\alpha_4^n$ (respectively $\alpha_6^n$) in the interior. Therefore, we have uniform curvature estimates for $\Sigma_n$ up to $\alpha_4\cup\alpha_6.$ 
 
Hence, for every compact contained in $\{z>0\}\cap \mathcal R,$ there exists a subsequence of $\Sigma_n$ that converges to a minimal surface. Taking an exhaustion by compact sets and using a diagonal process, we conclude that there exists a subsequence of $\Sigma_n$ that converges to a minimal surface $\Sigma$ that has $\alpha_4\cup\alpha_6$ in its boundary. From now on we will use the notation $\Sigma_n$ for this subsequence.

It remains to prove that in fact $\Sigma$ is a minimal surface with boundary $P_\infty.$ In order to do that, we will use the fact that each $\Sigma_n$ is a vertical graph in the interior. Let us denote by $u_n$ the function defined in int$\Omega(n)$ such that $\Sigma_n=\mbox{Graph}(u_n),$ where $\Omega(n)=\{(x,y,0): 0\leq x\leq n; 0\leq y\leq a\}.$

\begin{claim}
$u_{n-1}<u_n$ in {\rm int}$\Omega(n-1).$
\end{claim}
\begin{proof} Recall that each $\Sigma_n$ is the limit of a sequence of minimal graphs $\widetilde{\Sigma}_n^{\delta}=\mbox{Graph}(w_n^\delta)$ whose boundary is transversal to the Killing field $\partial x.$ Using the flow of the Killing field $\partial_x,$ we can prove that each $\widetilde{\Sigma}_n^\delta$ is above $\Sigma_{n-1},$ and then the limit surface $\Sigma_n$ has to be above $\Sigma_{n-1}.$ In fact, $\Sigma_n$ is strictly above $\Sigma_{n-1}$ in the interior, because as $\Sigma_n$ and $\Sigma_{n-1}$ are minimal surfaces, if they intersect at an interior point, there will be points of $\Sigma_n$ under $\Sigma_{n-1},$ and we already know that, by the property of $\widetilde{\Sigma}_n^\delta,$ this is not possible. 
\end{proof}

\begin{claim}
There are uniform gradient estimates for $\{u_n\}$ for points in $\alpha_1\cup \alpha_3.$
\end{claim}
\begin{proof} We will use the same idea as in Claim \ref{claim1}. For $y_0>a$ and $\delta>0$ consider the vertical strip bounded by $\beta_1=\{(x,y_0,c_0): d\leq x\leq d+\delta\},$ $\beta_2=\{(t,y_0,\frac{c_0}{d}t):0\leq t\leq d\},$ $\beta_3=\{(t+\delta,y_0,\frac{c_0}{d}t):0\leq t\leq d\}$ and $\beta_4=\{(x,y_0,0):0\leq x\leq \delta\}.$ This is a minimal surface transversal to the Killing field $\partial_y,$ hence any small perturbation of its boundary gives a minimal surface with that perturbed boundary. Thus, if we consider a small perturbation of the boundary of this vertical strip by just perturbing a little bit $\beta_1$ by a curve contained in $\{y\leq y_0\}$ joining the points $(d,y_0,c_0)$ and $(d+\delta,y_0,c_0),$ we will get a minimal surface $S$ with this perturbed boundary. This minimal surface $S$ will have the property that the tangent planes at the interior points of $\beta_4$ are not vertical, by the maximum principle with boundary.

Applying translations along the $x$-axis and $y$-axis, we can use the translates of $S$ to show that $\Sigma_n$ is under $S$ in a neighborhood of $\alpha_3,$ and then we have uniform gradient estimates  for points in $\alpha_3$. Analogously, constructing similar barriers, we can prove that we have uniform gradient estimates in a neighborhood of $\alpha_1.$ 
\end{proof}

Observe that besides the gradient estimates, the translates of the minimal surface $S$ form a barrier for points in a neighborhood of $\alpha_1\cup \alpha_3.$

We have that $\Sigma_n$ is a monotone increasing sequence of minimal graphs with uniform gradient estimates in $\alpha_1\cup \alpha_3,$ and it is a bounded graph for points in $\{x\geq d\}$ (because of the barrier given by the annulus $\mathcal A$). Therefore, there exists a subsequence of $\Sigma_n$ that converges to a minimal surface ${\widetilde \Sigma}$ with $\alpha_1\cup \alpha_3$ in its boundary. As we already know that $\Sigma_n$ converges to the minimal surface $\Sigma,$ we conclude that in fact $\Sigma=\widetilde{\Sigma},$ and then $\Sigma$ is a minimal surface with $\alpha_1\cup\alpha_3\cup\alpha_4\cup\alpha_6$ in its boundary. Notice that we can assume that $\Sigma$ has $P_\infty$ as its boundary, with $\Sigma$ being of class $C^1$ up to $P_\infty\setminus \{(0,0,0),(0,a,0)\}$ and continuous up to $P_\infty.$ The expected ``singly periodic Scherk minimal surface'' is obtained by rotating recursively $\Sigma$ by an angle $\pi$ about the vertical and horizontal geodesics in its boundary.

\begin{theorem}
In any semidirect product $\mathbb R^2\rtimes_A \mathbb R,$ where $A=\left(
\begin{array}{ccc}
	0 & b\\
	c&0
\end{array}\right),$ there exists a periodic minimal surface similar to the singly periodic Scherk minimal surface in $\mathbb R^3.$
\end{theorem}

\begin{flushleft}
\textsc{Instituto Nacional de Matem\' atica Pura e Aplicada (IMPA)}

\textsc{Estrada Dona Castorina 110, 22460-320, Rio de Janeiro-RJ, Brazil}

\textit{Email adress:} anamaria@impa.br
\end{flushleft}

\end{document}